\documentclass[12pt, reqno]{amsart}
\usepackage{amsmath,amssymb}
\usepackage[english]{babel}
\usepackage[font=small, width=.9\linewidth]{caption}
\usepackage{graphicx}
\usepackage{float}
\usepackage{pgf,tikz}
\usetikzlibrary{graphs,graphs.standard,calc}
\usepackage{upgreek}
\usepackage{mathrsfs}
\usepackage{soul}

\usepackage{color, xcolor}
\usepackage[colorlinks=true, linkcolor=blue, urlcolor=blue, citecolor=blue]{hyperref}
\usepackage[capitalise, nameinlink, noabbrev]{cleveref}
\usepackage{thmtools}

\usepackage{booktabs} % For tables
\usepackage{wrapfig} % For figures
\usepackage{enumitem} % For lists

\usepackage{subfigure}

\definecolor{myorange}{RGB}{255, 160, 18}

\definecolor{mygreen}{RGB}{50, 200, 50}

\newcommand{\ZZ}{\mathbb{Z}}

\newcommand{\KK}{\mathbb{K}}

\newcommand{\set}[1]{\left\{ #1 \right\}}
\newcommand{\mset}[1]{\left\{ #1 \right\}}

\newcommand{\floor}[1]{\left\lfloor {#1} \right\rfloor}

\newcommand{\overbar}[1]{\mkern 1.5mu\overline{\mkern-1.5mu#1\mkern-1.5mu}\mkern 1.5mu} % Middle ground between \bar and \overline
\renewcommand*\complement[1]{\overbar{#1}} % Bar over a set to denote complement
\newcommand*\complete[1]{\tilde{#1}} % Tilde over a graph to denote the complete graph on its vertices

\newcommand{\heightoperator}{\operatorname{height}}
\newcommand{\height}[1]{\heightoperator\left(#1\right)}

\newcommand{\cmi}[1]{\mathcal{C}\Big(#1;t_1,t_2\Big)} % Multidegree of ideal
\newcommand{\cmg}[1]{\mathcal{C}\left(R(#1);t_1,t_2\right)} % Multidegree of graph
\newcommand{\om}[1]{\mathcal{C}\left(R/I_S;t_1,t_2\right)}
\newcommand{\minS}[1]{\operatorname{\mathcal{S}_{min}}(#1)}
\newcommand{\minh}{h_{\operatorname{min}}}
\newcommand{\leftovers}[2]{N(\complement{#1})}
\newcommand{\components}[1]{c(#1)}

\newtheorem{theorem}{Theorem}[section]

\newtheorem{proposition}[theorem]{Proposition}

\newtheorem{headthm}{Theorem}

\theoremstyle{definition}

\newtheorem{example}[theorem]{Example}
\newtheorem{notation}[theorem]{Notation}

\newtheorem{remark}[theorem]{Remark}

\begin{document}
\title{MULTIDEGREES OF BINOMIAL EDGE IDEALS}
\author[J. Cooper]{Jacob Cooper}
\email{jcoope39@asu.edu}
\author[E. Leventhal]{Ethan Leventhal}
\email{eylevent@asu.edu}
\address{School of Mathematical and Statistical Sciences, Arizona State University, P.O. Box 871804, Tempe, AZ 85287-18041}

\begin{abstract}
    Let $G$ be a simple graph with binomial edge ideal $J_G$. We prove how to calculate the multidegree of $J_G$ based on combinatorial properties of $G$. In particular, we study the set $\minS{G}$ defined as the collection of subsets of vertices whose prime ideals have minimum codimension. We provide results which assist in determining $\minS{G}$, then calculate $\minS{G}$ for star, horned complete, barbell, cycle, wheel, and friendship graphs, and use the main result of the paper to obtain the multidegrees of their binomial edge ideals.
\end{abstract}
% \keywords{Multidegree, Binomial edge ideal, Graph theory}
\maketitle
% \tableofcontents

%%%%%%%%%%%%%%%%%%%%%%%%%%%%%%%%%%%%%%%%%%%%%%%%%%%%%%%%%%%%
\section{Introduction}\label{sec:intro}
%%%%%%%%%%%%%%%%%%%%%%%%%%%%%%%%%%%%%%%%%%%%%%%%%%%%%%%%%%%%
Let $n \in \mathbb{N}$ and let $G$ be a simple, connected graph with vertex set $V(G) = [n] := \set{1,\ldots,n}$ and edge set $E(G)$. Let $\KK$ be a field and let $T$ be the polynomial ring $T := \mathbb{K}[x_1,\ldots,x_n,y_1,\ldots,y_n]$. The \textit{binomial edge ideal} of $G$ is defined as
$$J_G := (x_iy_j-x_jy_i : i<j \text{ and } \set{i,j} \in E(G)).$$
These ideals have been intensely studied during the past few decades, in particular because of their connection to the class of ideals of 2-minors of a generic $2 \times n$ matrix, as well as their connection to conditional independence \cite{bolognini,ene,herzog,kiani}. Similarly, the multidegree is a potent geometric invariant that has been used across domains of algebra, geometry, and even statistics \cite{caminata,castillo,herrmann,knutson, michalek}. We begin this section with the necessary notations to state the main result, then we include the statement of our main theorem, and finally we provide an outline of the rest of the paper.

\begin{notation}\label{not:minS}
    Let $G$ be a graph. For a subset $S\subseteq V(G)$, let $\complement{S}$ denote $V(G) \setminus S$ and consider the induced subgraph $G[\complement{S}]$. Let $\components{\complement{S}}$ be the number of connected components of $G[\complement{S}]$. We denote $\minh := \min_{S \in V(G)}\{|S|-c(\complement{S})\}$. Then we define
    $$\minS{G} := \set{S \subseteq V(G): |S| - c(\complement{S}) = \minh}.$$
\end{notation}

By labeling the connected components $G_i$ for $i = 1, \dots, c(\complement{S})$, we establish the notation $\leftovers{S}{G} := \bigl(|V(G_1)|, \dots, |V(G_{c(\complement{S})})|\bigr)$ as the multiset containing the number of vertices of each connected component. The following is our main theorem.
\begin{headthm}[\cref{thm:multidegree}] \label{thmA}
    The multidegree of the binomial edge ideal of a graph $G$ is
    $$\cmg{G} = \sum_{S\in\minS{G}}\left[(t_1t_2)^{|S|}\cdot\prod_{n\in\leftovers{S}{G}}\frac{t_1^n-t_2^n}{t_1-t_2}\right].$$
\end{headthm}

This theorem makes the computation of the multidegree of $J_G$ very straightforward once $\minS{G}$ is known. 

We begin \cref{sec:background} with some background information on multidegrees. We define the ring of a graph $G$ to be $R(G) := T/J_G$ so that we can follow the notation of Miller and Sturmfels in \cite{miller} and denote the multidegree of $J_G$ with $\cmg{G}$ . The multidegree is a polynomial with integer coefficients, which can be interpreted in different ways. Geometrically, for example, its coefficients correspond to the number of intersection points between the variety determined by the ideal and a suitable product of general hyperplanes in a product of projective spaces; algebraically, these coefficients can be read from the top degree of the Hilbert polynomial of the ring \cite{castillo,vanderwaerden}.

The ideal $J_G$ is radical. In \cite{herzog},  Herzog, Hibi, Hreinsdottir, Kahle, and Rauh give a complete description of the prime decomposition of $J_G$; in \cref{sec:algebra}, we provide the algebraic tools that allow us to use this description to find multidegrees. In  \cref{sec:main}, we compute some foundational multidegrees that serve as `building blocks' for our final computations. Then we arrive at our main result, \cref{thm:multidegree}, which is a closed formula for the multidegree of $J_G$ that only involves combinatorial properties of the graph.

We provide some results to speed up the process of determining $\minS{G}$ and then spend \cref{sec:examples} demonstrating the usefulness of these results by calculating the multidegrees for six different families of graphs. Finally, \cref{sec:conclusion} concludes the paper with closing thoughts and ideas for further research.

%%%%%%%%%%%%%%%%%%%%%%%%%%%%%%%%%%%%%%%%%%%%%%%%%%%%%%%%%%%%
\section{Background Information}\label{sec:background}
%%%%%%%%%%%%%%%%%%%%%%%%%%%%%%%%%%%%%%%%%%%%%%%%%%%%%%%%%%%%
The multidegree was originally defined by van der Waerden in 1929 \cite{vanderwaerden}. It carries tremendous information about the algebraic and geometric properties of an ideal and its associated variety. This has made it a useful tool to approach multiple problems in algebra, geometry, and combinatorics \cite{caminata, castillo, herrmann, knutson, michalek}. Here we use the $\ZZ^2$-grading of $T$ induced by $x_i = (1,0)$ and $y_i = (0,1)$ that carries over the variables $t_1,t_2 = t_1,t_2$. We follow the notation from \cite{miller} and denote the Hilbert series of a $\ZZ^2$-graded module $M$ by $H(M;t_1,t_2)$. By \cite[Theorem 8.20]{miller}, there exists a \textit{$K$-polynomial} $\mathcal{K}(M;t_1,t_2)$ such that
$$H(M;t_1,t_2)  = \frac{\mathcal{K}(M;t_1,t_2)}{(1-t_1)^n(1-t_2)^n}.$$
We define the \textit{multidegree} of a $\ZZ^2$-graded module $M$ to be the sum $\mathcal{C}(M;t_1,t_2) \in \ZZ[t_1,t_2]$ of all terms in $\mathcal{K}(M;1-t_1,1-t_2)$  of smallest degree, i.e., those of total degree $\text{codim}(M):=\dim(T)-\dim(M)$. Therefore, one has $\mathcal{K}(T;t_1,t_2)=1$.
Despite the popularity of binomial edge ideals and the usefulness of the multidegree, to the best of or knowledge there is a lack of literature studying both together. Our paper aims to investigate connections between the multidegree of a binomial edge ideal and  the combinatorial information of its underlying graph. We note that a recent paper by Kumar and Sarkar studies the Hilbert series of decomposable graphs in the $\ZZ$-graded setting \cite{kumar}.

The final piece of introductory information, found in \cite{herzog}, is the computation of the prime decomposition of $J_G$. Recall that we use $\components{\complement{S}}$ to denote the number of connected components of $G[\complement{S}]$. Call these components $G_1,\ldots,G_{\components{\complement{S}}}$. For each $G_i$, let $\complete{G_i}$ be the complete graph on $V(G_i)$. Then for a subset $S \subseteq V(G)$ we define the prime ideal
$$P_S(G) := \left(\bigcup_{i\in S}\set{x_i,y_i},J_{\complete{G}_1},\ldots,J_{\complete{G}_{\components{\complement{S}}}}\right).$$
In \cite[Theorem 3.2]{herzog}, it is shown that $J_G = \bigcap_{S \subseteq V(G)}P_S(G)$. We expand on this theorem in the following sections to obtain a direct calculation of the multidegree of $J_G$.

%%%%%%%%%%%%%%%%%%%%%%%%%%%%%%%%%%%%%%%%%%%%%%%%%%%%%%%%%%%%
\section{Relationships Between Ideals and Multidegrees}\label{sec:algebra}
%%%%%%%%%%%%%%%%%%%%%%%%%%%%%%%%%%%%%%%%%%%%%%%%%%%%%%%%%%%%
In this section we include some preparatory results that simplify the computation of the multidegree of  $J_G$. Our first proposition narrows down which multidegrees we need for our calculations. It is proven in \cite{herzog} that the height (codimension) of $P_S(G)$ for a subset $S \subseteq V(G)$ is given by $\height{P_S(G)} = |S|+n-\components{\complement{S}}$ \cite[Lemma 3.1]{herzog}. This makes it clear that $\minS{G}$, defined in \cref{not:minS}, is the  collection of all subsets of $V(G)$ whose corresponding prime ideals have minimum height.

\begin{proposition}\label{prop:only min heights}
    For a graph $G$ we have
    $$\cmi{T/\bigcap_{S \subseteq V(G)}P_S(G)} = \cmi{T/\bigcap_{S\in\minS{G}}P_S(G)}.$$
    
   Moreover, if $I_1,\ldots,I_m$ are homogeneous ideals of $T$ with the same height, then
    
    \[\cmi{R/\bigcap_{i\in[m]} I_i} = \sum_{i\in[m]}\left(\cmi{R/I_i}\right).\] 
\end{proposition}
\begin{proof}
    These follow from \cite[Definition 8.43, Theorem 8.53, and Corollary 8.54]{miller}. 
\end{proof}

In other words, the only prime ideals that contribute to the multidegree of $J_G$ are those with minimum height. Additionally, this allows us to work with the sum of multidegrees instead of an intersection of ideals.

Next, we present a similar proposition for the sum of ideals.

\begin{proposition}\label{prop:product of multidegrees}
    Let $I_1,\ldots, I_n$ be homogeneous ideals of $T$ such that for distinct $i$ and $j$, the variables used in a minimal set of generators of $I_i$ are not used in that of $I_j$. Moreover, assume all the variables in $T$ are used in the ideals $I_i$. Let $T_i$ be the polynomial ring in the variables involved in the generators of $I_i$. Then
    $$\mathcal{C}\left(T/\sum_{i\in[n]} I_i;t_1,t_2\right) = \prod_{i\in[n]}(\mathcal{C}(T_i/(I_i\cap T_i);t_1,t_2)).$$
\end{proposition}
\begin{proof}
    By assumption one has a graded isomorphism
    $$\frac{T}{\sum_{i\in[n]} I_i}\cong \frac{T_1}{I_1\cap T_1} \otimes_\KK\cdots\otimes_\KK \frac{T_n}{I_n\cap T_n}.$$
    Therefore 
    $$H\Big(T/\sum_{i\in[n]}I_i;t_1,t_2\Big)
	=\frac{\prod_{i\in[n]}\mathcal{K}(T_i/(I_i\cap
        T_i);t_1,t_2)}{(1-t_1)^{n}(1-t_2)^n}$$
    and the result follows.
\end{proof}

The following proposition computes the multidegree of ideals of variables.

\begin{proposition}\label{prop:other}
    For a graph $G$ and subset of vertices $S$,
    $$\mathcal{C}(T/\bigcup_{i \in S}\set{x_i,y_i};t_1,t_2) = (t_1t_2)^{|S|}.$$ 
\end{proposition}
\begin{proof}
   Note that 
   $$H(T/\bigcup_{i \in S}\set{x_i,y_i});t_1,t_2)
   	=\frac{1}{(1-t_1)^{n-|S|}(1-t_2)^{n-|S|}}
        =\frac{(1-t_1)^{|S|}(1-t_2)^{|S|}}{(1-t_1)^{n}(1-t_2)^n}.$$ 
   The result clearly follows.
\end{proof}

\begin{remark}\label{rem:only connected}
    Let $G_1$ and $G_2$ be two disjoint graphs and suppose $H$ is their union. Then by \cref{prop:product of multidegrees}, $\cmg{H} = \cmg{G_1}\cdot\cmg{G_2}$, meaning the multidegree of the binomial edge ideal of the union of two disjoint graphs is the product of their multidegrees. For this reason, we assume all graphs are connected graphs for the rest of the paper.
\end{remark}

%%%%%%%%%%%%%%%%%%%%%%%%%%%%%%%%%%%%%%%%%%%%%%%%%%%%%%%%%%%%
\section{Computing the Multidegree of the Binomial Edge Ideal of a Graph}\label{sec:main}
%%%%%%%%%%%%%%%%%%%%%%%%%%%%%%%%%%%%%%%%%%%%%%%%%%%%%%%%%%%%
In \cref{sec:algebra} we showed how the prime decomposition of $J_G$ can be used to find the multidegree $\cmg{G}$. We see from the definition that each $P_S(G)$ is constructed from two types of ideals, so in this section we first obtain the two corresponding multidegrees.

The first multidegree we need is that of an ideal of the form $\bigcup_{i \in S}\set{x_i,y_i}$, which we showed in \cref{prop:product of multidegrees} to be
$$\mathcal{C}(T/\bigcup_{i \in S}\set{x_i,y_i};t_1,t_2) = (t_1t_2)^{|S|}.$$

The other multidegree that we need is the multidegree of the binomial edge ideal of a complete graph. For this we turn to a theorem proved by Bruns, Conca, Raicu, and Varbaro in \cite{bruns}: Let $\mathbb{K}$ be a field, $X$ be the $n \times m$ matrix of variables and $2\leq t\leq\text{min}(n,m)$. For a subset $L$ of elements of a ring, define $h_v(L)$ to be the sum of all power products of elements of $L$ whose exponents sum to $v$. Let $I_t$ be the ideal generated by $t$-minors of $X$. Then the multidegree of the ring $R/I_t(X)$ with respect to the $\ZZ^m$-graded structure given by columns is $$\cmi{R/I_t(X)}=(t_1,\ldots,t_m)^{t-2}\det(h_{\ell-i+j}(t_1,\ldots,t_m))_{i,j=1,2,\ldots,t-1},$$ where $\ell=n+1-t$ \cite[Theorem 4.6.8]{bruns}. Therefore in our setting we obtain the following proposition.

\begin{proposition}\label{prop:complete}
    The multidegree of the binomial edge ideal of a complete graph $K_n$ is
    $$\cmg{K_n} = \frac{t_1^n-t_2^n}{t_1-t_2}.$$
\end{proposition}
\begin{proof}
    The binomial edge ideal of a complete graph on $n$ vertices is simply the ideal generated by the $2$-minors of a $2 \times n$ matrix of variables. This is an application of \cite[Theorem 4.6.8]{bruns} where $n$ is the order of the complete graph and $m=2$. Following the theorem, we calculate $t=2$ and $\ell=n+1-2=n-1$. Then the multidegree is
    \begin{align*}
	\cmg{K_n}
	= \text{det}(h_{n-1}(t_1,t_2)) 
	= \sum_{\substack{i+j=n-1 \\ i,j \geq 0}}t_1^it_2^j 
	= \frac{t_1^n-t_2^n}{t_1-t_2}.
    \end{align*}
\end{proof}

Recall that to compute $J_G$, we need to find $J_{\complete{G}_1}, \ldots, J_{\complete{G}_{c(\complement{S})}}$. Since $\complete{G}_1, \ldots, \complete{G}_{c(\complement{S})}$ are complete graphs, we can use \cref{prop:complete} to find their multidegrees very easily once we know the orders of these complete graphs, i.e., the number of vertices in the connected components of $G[\complement{S}]$. This motivates our notation $\leftovers{S}{G} := \bigl(|V(G_1)|, \dots, |V(G_{c(\complement{S})})|\bigr)$.

Now we present the main result of the paper, a formula to combinatorially compute the multidegree of the binomial edge ideal of a graph $G$. Since it relies on determining $\minS{G}$, we follow the theorem with a discussion of shortcuts for finding $\minS{G}$.

\begin{theorem}\label{thm:multidegree}
    For a graph $G$, the multidegree of $J_G$ is
    $$\cmg{G} = \sum_{S\in\minS{G}}\left[(t_1t_2)^{|S|}\cdot\prod_{n\in\leftovers{S}{G}}\frac{t_1^n-t_2^n}{t_1-t_2}\right].$$
\end{theorem}
\begin{proof}
    Let $G$ be a simple, connected graph. We have the following chain of equalities:
    \begin{align*}
        &\cmg{G} \\
        &= \cmi{T/J_G} &(\text{By definition}) \\
        &= \cmi{T/\bigcap_{S \subseteq V(G)}P_S(G)} &\text{\cite[Theorem 3.2]{herzog}} \\
        &= \cmi{T/\bigcap_{S \in \minS{G}}P_S(G)} &(\text{\cref{prop:only min heights}}) \\
        &= \sum_{S \in \minS{G}}\cmi{T/P_S(G)} &(\text{\cref{prop:only min heights}}) \\
        &= \sum_{S \in \minS{G}}\cmi{ T/\Big(\bigcup_{i\in S}\set{x_i,y_i},J_{\complete{G}_1},\ldots,J_{\complete{G}_{\components{\complement{S}}}}\Big) } &(\text{By definition}) \\
        &= \sum_{S \in \minS{G}}\left[\cmi{T/ \bigcup_{i\in S}\set{x_i,y_i}} \cdot \prod_{i=1,\ldots,\components{\complement{S}}}\left(\cmg{\complete{G}_i}\right)\right] &(\text{\cref{prop:product of multidegrees}}) \\
        &= \sum_{S\in\minS{G}}\left[(t_1t_2)^{|S|}\cdot\prod_{n\in\leftovers{S}{G}}\frac{t_1^n-t_2^n}{t_1-t_2}\right]. &\tag*{\llap{(\text{\cref{prop:other} and \cref{prop:complete}})}}
    \end{align*}
\end{proof}

\cref{thm:multidegree} explicitly states how to compute the multidegree of $J_G$ once we know $\minS{G}$. The obvious goal now is to determine $\minS{G}$. We could calculate $\height{P_S(G)}$ for all $S$, but since $\minS{G}$ is a subset of the power set $\mathcal{P}(V(G))$, this quickly becomes impractical. Here we provide our most relevant results to help find $\minS{G}$ efficiently, starting with some book-keeping notation.

\begin{notation}\label{not:minus}
    For a set $X$ and an element $y$ we denote $X+y := X \cup \set{y}$ and $X-y := X \setminus \set{y}$. Similarly, for a graph $G$ and vertex $v$ we denote $G-v := G[V(G)-v]$.
\end{notation}

Our first result explains how \textit{cut vertices} (vertices $v$ for which $\components{G-v} > \components{G}$) affect the height of prime ideals corresponding to the relevant subsets.

\begin{remark}\label{rem:cut vertex}
    Let $G$ be a graph,  $S \subseteq V(G)$, and $v \in V(G) \setminus S$. From \cite[Lemma 3.1]{herzog} we get the following:
    \begin{enumerate}[label=\alph*)]
        \item If $v$ is not a cut vertex of $G[\complement{S}]$ then $\height{P_{S+v}(G)} > \height{P_S(G)}$.
        \item If $v$ is a cut vertex $G[\complement{S}]$ which cuts a component of $G[\complement{S}]$ into exactly two connected components then $\height{P_{S+v}(G)} = \height{P_S(G)}$.
        \item If $v$ cuts a component of $G[\complement{S}]$ into three or more connected components then $\height{P_{S+v}(G)} < \height{P_S(G)}$.
    \end{enumerate}
\end{remark}

Next, recall that a \textit{simplicial vertex} of a graph is a vertex whose open neighborhood forms a clique, i.e., whose neighbors are all adjacent. This includes vertices of degree 1, called \textit{leaves}. We prove that simplicial vertices can be ignored when finding $\minS{G}$.

\begin{proposition}\label{prop:simplicial}
    Let $G$ be a graph. For all $S\in\minS{G}$, $S$ contains no simplicial vertices.
\end{proposition}
\begin{proof}
    Let $G$ be a graph with simplicial vertex $v$ and let $v \not\in S \subseteq V(G)$. We will show $S+v\notin\minS{G}$. Since $v$ is a simplicial vertex, its neighbors in $G[\complement{S}]$ are adjacent (regardless of $S$). This means $v$ is not a cut vertex of $G[\complement{S}]$. By \cref{rem:cut vertex}, $\height{P_{S+v}(G)} > \height{P_S(G)}$. We have found a subset of vertices whose prime ideal has height less than that of $S+v$, thus $S+v\notin\minS{G}$.
\end{proof}

The last result we present here requires $S \in \minS{G}$ to be a \textit{separating set}, a vertex subset of a (connected) graph $G$ whose removal disconnects $G$.

\begin{proposition}\label{prop:separating}
    Let $G$ be a graph. For all nonempty $S\in\minS{G}$, $S$ is a separating set of $G$.
\end{proposition}
\begin{proof}
    Let $G$ be a graph of order $n$ and suppose $S\in\minS{G}$ is nonempty. For all $T \subseteq V(G)$ we have $\height{P_S(G)} \leq \height{P_T(G)}$, so in particular $\height{P_S(G)} \leq \height{P_{\varnothing}(G)}$. By \cite[Lemma 3.1]{herzog},
    \begin{align*}
        |S|+n-c(\complement{S}) &\leq n-1 \\
        |S|+c(\complement{S}) &\leq -1 \\
        |S|+1 &\leq c(\complement{S}).
    \end{align*}
    Since $S$ is nonempty, $|S| \geq 1$. This implies $2 \leq c(\complement{S})$, thus $S$ is a separating set of $G$.
\end{proof}

%%%%%%%%%%%%%%%%%%%%%%%%%%%%%%%%%%%%%%%%%%%%%%%%%%%%%%%%%%%%
\section{Example Applications}\label{sec:examples}
%%%%%%%%%%%%%%%%%%%%%%%%%%%%%%%%%%%%%%%%%%%%%%%%%%%%%%%%%%%%
To demonstrate the usefulness of \cref{thm:multidegree}, we take this section to work through some examples involving basic families of graphs. First, a bit of notation.

\begin{notation}\label{not:copies}
    Recall that $\leftovers{S}{G}$ records the sizes of the connected components of $G[\complement{S}]$. When describing $\leftovers{S}{G}$ in these examples, we will use the convenient notation of $n:x$ to denote $n$ copies of $x$.
\end{notation}

The first example we consider is a star graph, which gives a great first look at our proposed process for calculating multidegree, as well as showing off \cref{prop:simplicial} for reductions in cases with simplicial vertices.

\begin{example}[Star Graph]
    Let $S_n$ be a \textit{star graph} of order $n$, so it has $n-1$ leaves and one central vertex which we will call $v_0$ (as in \cref{fig1}). We assume $n>3$, since when $n \leq 3$, $S_n$ is a path graph  and should be treated as such; see \cref{prop:path_graphs}.   The first step in finding the multidegree is determining $\minS{S_n}$. By \cref{prop:simplicial} we know $\minS{S_n}$ does not include subsets of vertices that contain a leaf vertex. This rules out all but the two subsets $\varnothing$ and $\set{v_0}$. By \cite[Lemma 3.1]{herzog}, $\height{P_{\varnothing}(G)}=0+n-1=n-1$ and $\height{P_{\set{v_0}}(G)}=1+n-(n-1)=2$. Since $n>3$, we can see that $P_{\set{v_0}}(G)$ has a lower height, so $\minS{S_n} = \set{\set{v_0}}$. We have $\leftovers{\set{v_0}}{S_n}=\mset{n-1:1}$, thus by \cref{thm:multidegree},
      $$\cmg{S_n} = (t_1t_2)^1\cdot\left(\frac{t_1^1-t_2^1}{t_1-t_2}\right)^{n-1} = t_1t_2.$$
\end{example}

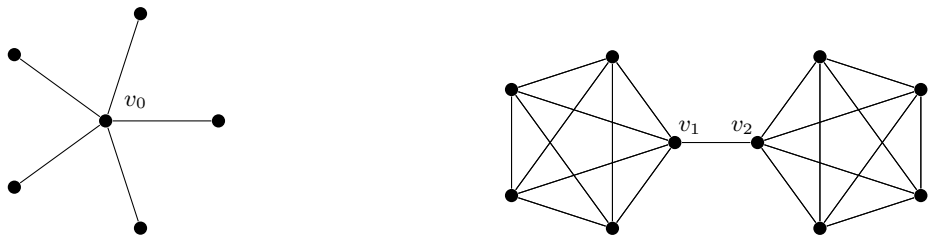
\begin{figure}[ht]
    \centering
    \mbox{
        \subfigure{
    	\begin{tikzpicture} % Star graph
                [node/.style={circle,fill=black, inner sep=0pt,minimum size=5pt},label/.style={circle, inner sep=0pt,minimum size=5pt}]
    		\node at (360:0mm) [node] (center) {};
    		\foreach \n in {1,...,5}{
    			\node at ({\n*360/5}:1.5cm) [node] (n\n) {};
    			\draw (center) -- (n\n);
    		}
    		\begin{scriptsize}
    			\draw (0.4,0.25) node {$v_0$};
    		\end{scriptsize}
    	\end{tikzpicture}
        }
        \qquad\qquad\qquad\qquad
        \subfigure{
    	\begin{tikzpicture} % Barbell graph
            [node/.style={circle,fill=black, inner sep=0pt,minimum size=5pt},label/.style={circle, inner sep=0pt,minimum size=5pt}]
            \foreach \i in {1,...,5}{
                \node at ({(\i-1)*360/5}:1.2cm) [node] (n\i) {};
                \node at ($({(\i-1)*360/5+180}:1.2cm)+(3.5,0)$) [node] (m\i) {};
            }
            \foreach \i in {1,...,5}{
                \foreach \j in {1,...,5}{
                    \ifnum\i=\j
                    \else
                        \draw (n\i) -- (n\j);
                        \draw (m\i) -- (m\j);
                    \fi
                }
            }
            \draw (n1)--(m1);
            \begin{scriptsize}
                \draw (1.4,0.2) node {$v_1$};
                \draw (2.1,0.2) node {$v_2$};
            \end{scriptsize}
            \end{tikzpicture}
        }
    }
    \caption{$S_6$ (Left) and $B_5$ (Right)} \label{fig1}
\end{figure}

Now we take a look at barbell graphs to prove that their multidegrees turn out to be quite interesting: their coefficients are consecutive odd numbers.

\begin{example}[Barbell Graph]
    We will use the notation of $B_n$ to denote the \textit{barbell graph} constructed by taking two copies of $K_n$ and connecting a vertex $v_1$ in one to a vertex $v_2$ in the other (such as in \cref{fig1}). \cref{prop:simplicial} rules out all simplicial vertices, yielding $\minS{B_n} \subseteq \mathcal{P}(\set{v_1,v_2})$. Doing the height calculations for the four possible elements reveals that $\minS{B_n} = \set{\varnothing,\set{v_1},\set{v_2}}$.
    
    We find that $\leftovers{\varnothing}{B_n}=\mset{2n}$ and $\leftovers{\set{v_1}}{B_n}=\leftovers{\set{v_2}}{B_n}=\mset{n,n-1}$. Then by \cref{thm:multidegree},
    
    \begingroup
        \allowdisplaybreaks
        \begin{align*}
            \cmg{B_n} &= (t_1t_2)^0\cdot\frac{t_1^{2n}-t_2^{2n}}{t_1-t_2} + 2\cdot\left[(t_1t_2)^1\cdot\frac{t_1^n-t_2^n}{t_1-t_2}\cdot\frac{t_1^{n-1}-t_2^{n-1}}{t_1-t_2}\right] \\
            &= \frac{t_1^{2n}-t_2^{2n}}{t_1-t_2} + 2t_1t_2\cdot\frac{t_1^n-t_2^n}{t_1-t_2}\cdot\frac{t_1^{n-1}-t_2^{n-1}}{t_1-t_2} \\
            &= \sum_{\substack{i+j=2n-1 \\ i,j \geq 0}}t_1^it_2^j + 2t_1t_2\cdot\sum_{\substack{i+j=n-1 \\ i,j \geq 0}}t_1^it_2^j\cdot\sum_{\substack{i+j=n-2 \\ i,j \geq 0}}t_1^it_2^j \\
            &= \sum_{\substack{i+j=2n-1 \\ i,j \geq 0}}t_1^it_2^j + 2t_1t_2\cdot\sum_{\substack{i+j=2n-3 \\ i,j \geq 0}}(\min\set{i,j}+1)t_1^it_2^j \\
            &= \sum_{\substack{i+j=2n-1 \\ i,j \geq 0}}t_1^it_2^j + \sum_{\substack{i+j=2n-1 \\ i,j \geq 0}}2\min\set{i,j}t_1^it_2^j \\
            &= \sum_{\substack{i+j=2n-1 \\ i,j \geq 0}}(1+2\min\set{i,j})t_1^it_2^j.
        \end{align*}
    \endgroup
    Thus the coefficients of $\cmg{B_n}$ are consecutive odd numbers!
\end{example}

Next we introduce a new family of graphs called the \textit{horned complete graphs}, which have a particularly interesting multidegree. They also illustrate \cref{rem:cut vertex}.

\begin{example}[Horned Complete Graph]\label{ex:horned}
    Let $\check{K}_n$ be a \textit{horned complete graph}, which we define as the complete graph $K_n$ on vertices $U:=\set{v_1,\ldots,v_n}$ with two leaves (`horns') attached to each (see \cref{fig2}). By \cref{prop:simplicial}, we know $\minS{\check{K}_n}$ will not include subsets that contain a `horn' vertex, which narrows it down to $\minS{\check{K}_n} \subseteq \mathcal{P}(\set{v_1,\ldots,v_n})$. To determine it exactly, first note that $v_1$ cuts the graph into three connected components, so by \cref{rem:cut vertex}, $\height{P_{\set{v_1}}(\check{K}_n)} < \height{P_{\varnothing}(\check{K}_n)}$. Now in $\check{K}_n - v_1$, we once again have a vertex $v_2$ that cuts one component into three, so $\height{P_{\set{v_1,v_2}}(\check{K}_n)} < \height{P_{\set{v_1}}(\check{K}_n)}$. This pattern continues until $\check{K}_n - \set{v_1,\ldots,v_{n-1}}$, in which $v_n$ cuts one component into two, so $\height{P_{\set{v_1,\ldots,v_n}}(\check{K}_n)} = \height{P_{\set{v_1,\ldots,v_{n-1}}}(\check{K}_n)}$. Because of the graph's symmetry, we can repeat the argument above for all $v_i \in U$ and conclude $\minS{\check{K}_n} = \set{U-v_i:i\in[n]} \cup \set{U}$.

    See that $\leftovers{U-v_i}{\check{K}_n}=\mset{2n-2:1,3}$ for $i\in[n]$ and $\leftovers{U}{\check{K}_n}=\mset{2n:1}$, so by \cref{thm:multidegree},
    \begin{align*}
        \cmg{\check{K}_n} &= n\cdot\left[(t_1t_2)^{n-1}\cdot\left(\frac{t_1^1-t_2^1}{t_1-t_2}\right)^{2n-2}\cdot\frac{t_1^3-t_2^3}{t_1-t_2}\right] + (t_1t_2)^n\cdot\left(\frac{t_1^1-t_2^1}{t_1-t_2}\right)^{2n} \\
        &= n\cdot\left[(t_1t_2)^{n-1}\cdot\frac{t_1^3-t_2^3}{t_1-t_2}\right] + (t_1t_2)^n \\
        &= n(t_1t_2)^{n-1}(t_1^2+t_1t_2+t_2^2) + t_1^nt_2^n \\
        &= n(t_1^{n+1}t_2^{n-1}+t_1^nt_2^n+t_1^{n-1}t_2^{n+1}) + t_1^nt_2^n \\
        &= nt_1^{n+1}t_2^{n-1}+(n+1)t_1^nt_2^n+nt_1^{n-1}t_2^{n+1}.
    \end{align*}
    Note that the leading coefficient of $\cmg{\check{K}_n}$ is $n$ rather than a constant, so it scales with the order of the graph. Based on our data collection in \cref{sec:conclusion}, this is a particularly rare property for the multidegrees of these ideals to have.
\end{example}

\begin{figure}[ht]
    \centering
    \mbox{
        \subfigure{
            \begin{tikzpicture} % Horned complete graph
                [node/.style={circle,fill=black, inner sep=0pt,minimum size=5pt},label/.style={circle, inner sep=0pt,minimum size=5pt}]
         	\foreach \i in {1,...,4}{
         		\node at ({(\i-1)*360/4+45}:1cm) [node] (v\i) {};
         		\node at ({(\i-1)*360/4+65}:1.8cm) [node] (n\i) {};
         		\node at ({(\i-1)*360/4+25}:1.8cm) [node] (m\i) {};
         		\begin{scriptsize}
         			\draw ({(\i-1)*360/4+20}:1.1cm) node {$v_{\i}$};
         		\end{scriptsize}
         	}
         	\foreach \i in {1,...,4}{
         		\draw (v\i) -- (n\i);
         		\draw (v\i) -- (m\i);
         		\foreach \j in {1,...,4}{
         			\ifnum\i=\j
         			\else
         			\draw (v\i) -- (v\j);
         			\fi
         		}
         	}
            \end{tikzpicture}
        }
        \qquad\qquad\qquad\qquad
        \subfigure{
            \begin{tikzpicture} % Cycle graph
                [node/.style={circle,fill=black, inner sep=0pt,minimum size=5pt},label/.style={circle, inner sep=0pt,minimum size=5pt}]
                \foreach \i in {1,...,5}{
                    \node at ({(\i-1)*360/5}:1.2cm) [node] (n\i) {};
                }
                \foreach \i in {1,...,4}{
                    \draw (n\i) -- (n\the\numexpr\i+1\relax);
                }
                \draw (n5) -- (n1);
                \begin{scriptsize}
                    \draw (1.4,0.2) node {$v$};
                \end{scriptsize}
            \end{tikzpicture}
        }
    }
    \caption{$\check{K}_4$ (Left) and $C_5$ (Right)} \label{fig2}
\end{figure}
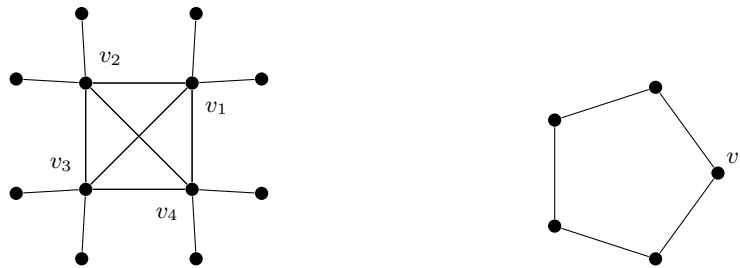

Next we get the multidegree of the binomial edge ideal of a cycle graph, which demonstrates the use of familiar subgraphs to determine $\minS{G}$.

\begin{example}[Cycle Graph]
    Let $C_n$ be a \textit{cycle graph} on $n$ vertices and let $v$ be an arbitrary vertex (as in \cref{fig2}). We have that $C_n-v$ is just $P_{n-1}$, a path graph with $n-1$ vertices;  see \cref{prop:path_graphs}. Note that for all $S \subset V(P_{n-1})$, $v \not\in S$ is either not a cut vertex of $G[\complement{S}]$ or cuts a component into exactly two components, so it follows from \cref{rem:cut vertex} that $\varnothing \in \minS{P_{n-1}}$. Hence $\varnothing \in \minS{C_n-v}$, so by definition, for all $S \subseteq V(C_n-v)$ we have
    $\height{P_S(C_n-v)} \geq \height{P_\varnothing(C_n-v)}$. Equivalently, for all $S \subseteq V(C_n)$ with $v \in S$ we have
    $\height{P_S(C_n)} \geq \height{P_{\set{v}}(C_n)}$. Notice $v$ is not a cut vertex of $C_n$, so by \cref{rem:cut vertex}, 
    $\height{P_{\set{v}}(C_n)} > \height{P_\varnothing(C_n)}$. Thus for all nonempty $S \subseteq V(C_n)$ we have 
    $\height{P_S(C_n)} > \height{P_{\varnothing}(C_n)}$. Hence $\minS{C_n} = \set{\varnothing}$.

    By \cref{thm:multidegree}, the multidegree is
    $$\cmg{C_n} = (t_1t_2)^0\cdot\frac{t_1^n-t_2^n}{t_1-t_2} = \frac{t_1^n-t_2^n}{t_1-t_2}.$$
\end{example}

Our second to last example is the wheel graph, which uses \cref{prop:separating}.

\begin{example}[Wheel Graph]
    Let $W_n$ be a \textit{wheel graph} of order $n$, which is a cycle on $n-1$ vertices with each vertex connected to a central vertex $v_0$ (such as in \cref{fig3}). Suppose $S \in V(G)$ is nonempty.
            
    (Case 1) Suppose $v_0 \in S$. Note $v_0$ is not a cut vertex of $W_n$, so by \cref{rem:cut vertex}, $\height{P_{\set{v_0}}(W_n)} > \height{P_{\varnothing}(W_n)}$. Also, $W_n-v_0$ is isomorphic to $C_{n-1}$ which has $\minS{C_{n-1}}=\set{\varnothing}$. By the same reasoning as the previous example, $\height{P_{S}(W_n)} > \height{P_{\varnothing}(W_n)}$, thus $S\notin\minS{W_n}$.
            
    (Case 2) Suppose $v_0 \notin S$. Then $S$ cannot be a separating set, so by \cref{prop:separating}, $S\notin\minS{W_n}$.
            
    In both cases $S\notin\minS{W_n}$, thus $\minS{W_n}=\set{\varnothing}$.
        
    Recall that $\varnothing{}$ was also the only element of $\minS{C_n}$. This makes the computation of their multidegrees quite simple, and the fact that $C_n$ and $W_n$ are both connected graphs with the same size means they have the same multidegree.
    $$\cmg{W_n} = (t_1t_2)^0\cdot\frac{t_1^n-t_2^n}{t_1-t_2} = \frac{t_1^n-t_2^n}{t_1-t_2}.$$
\end{example}

\begin{figure}[ht]
    \centering
    \mbox{
        \subfigure{
    	\begin{tikzpicture} % Wheel graph
            [node/.style={circle,fill=black, inner sep=0pt,minimum size=5pt},label/.style={circle, inner sep=0pt,minimum size=5pt}]
            \node at (360:0mm) [node] (center) {};
            \foreach \i in {1,...,5}{
                \node at ({(\i-1)*360/5}:1.2cm) [node] (n\i) {};
            }
            \foreach \i in {1,...,4}{
                \draw (center) -- (n\i);
                \draw (n\i) -- (n\the\numexpr\i+1\relax);
            }
            \draw (n5) -- (n1);
            \draw (center) -- (n5);
            \begin{scriptsize}
                \draw (0.4,0.25) node {$v_0$};
            \end{scriptsize}
            \end{tikzpicture}
        }
        \qquad\qquad\qquad\qquad
        \subfigure{
            \begin{tikzpicture} % Friendship graph
            [node/.style={circle,fill=black, inner sep=0pt,minimum size=5pt},label/.style={circle, inner sep=0pt,minimum size=5pt}]
            \node at (360:0mm) [node] (center) {};
            \foreach \i in {1,...,4}{
                \node at ({(\i-1)*360/4+20}:1.2cm) [node] (n\i) {};
                \node at ({(\i-1)*360/4-20}:1.2cm) [node] (m\i) {};
            }
            \foreach \i in {1,...,4}{
                \draw (center) -- (n\i);
                \draw (center) -- (m\i);
                \draw (n\i) -- (m\i);
            }
            \begin{scriptsize}
                \draw (0.4,0.35) node {$v_0$};
            \end{scriptsize}
            \end{tikzpicture}
        }
    }
    \caption{$W_6$ (Left) and $F_4 \cong W_{4:3}$ (Right)} \label{fig3}
\end{figure}
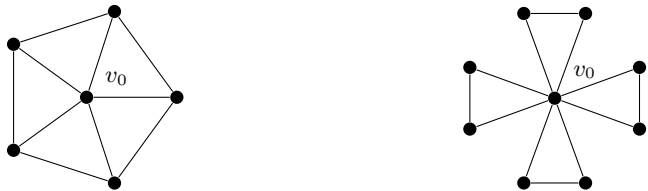

Finally, we demonstrate \cref{prop:simplicial} by computing the multidegree of the binomial edge ideals of friendship graphs and then briefly explain how the strategy can generalize to windmill graphs.

\begin{example}[Friendship Graph]
    Let $F_n$ be the \textit{friendship graph} constructed by connecting $n$ copies of $K_3$ at a common vertex $v_0$, as in \cref{fig3}. By \cref{prop:simplicial}, $\minS{F_n} \subseteq \mathcal{P}(\set{v_0})$. By \cref{rem:cut vertex} we see that $\height{P_{\set{v_0}}(F_n)} < \height{P_{\varnothing}(F_n)}$, so $\minS{F_n} = \set{\set{v_0}}$.
    
    We compute $\leftovers{\set{v_0}}{F_n} = \mset{n:2}$. By \cref{thm:multidegree}, 
    \begin{align*}
        \cmg{F_n} &= (t_1t_2)^1\cdot\left(\frac{t_1^2-t_2^2}{t_1-t_2}\right)^n \\
        &= t_1t_2\cdot(t_1+t_2)^n \\
        &= t_1t_2\cdot\sum_{\substack{i+j=n \\ i,j \geq 0}}\binom{n}{i}t_1^it_2^j \\
        &= \sum_{\substack{i+j=n \\ i,j \geq 0}}\binom{n}{i}t_1^{i+1}t_2^{j+1}.
    \end{align*}
    This strategy works equally well for \textit{windmill graphs} $W_{n,m}$, a generalization of friendship graphs that combine $m$ copies of $K_n$. The only difference in calculations is that $\leftovers{\set{v_0}}{W_{n,m}} = \mset{m:n-1}$, so the algebra requires the multinomial theorem instead of the binomial theorem.
\end{example}

Let $P_n$ denote a {\it path graph} with $n$ vertices, a graph whose edges are $\{i,i+1\}$ for $1\leq i<n$. We were not able to obtain a satisfying closed formula for the multidegree of path graphs using our main theorem. However, since $J_{P_n}$ is generated by a regular sequence, in the following proposition we are able to make this  computation using alternative methods.

\begin{proposition}\label{prop:path_graphs}
    For every positive integer $n$ we have 
    $$\cmg{P_n} = \sum_{\substack{i+j=n-1 \\ i,j \geq 0}}\binom{n-1}{i}t_1^{i}t_2^{j}.$$
\end{proposition}
\begin{proof}
    Let $f_i = x_iy_{i+1}-x_{i+1}y_i$ for $1\leq i< n$ be the generators of $J_{P_n}$. Notice that each $f_i$ is homogeneous of  degree $(1,1)$ and that $f_1,\ldots, f_{n-1}$ form a  regular sequence. Therefore we have graded short exact sequences
    $$0\to T/(f_1,\ldots, f_{i-1})[(-1,-1)] \xrightarrow{f_i}  T/(f_1,\ldots, f_{i-1}) 
    \to T/(f_1,\ldots, f_{i}) 
    \to 0,$$
    for each $i$, where $(f_1,\ldots, f_{i-1})$ is the zero ideal if $i=1$. Thus 
    $$H\Big(T/(f_1,\ldots, f_{i}); t_1,t_2\Big)
        = (1-t_1t_2)H\Big(T/(f_1,\ldots, f_{i-1}); t_1,t_2\Big)$$
    for each $i$, and so
    $$ H\Big(R(P_n); t_1,t_2\Big)
        =  H\Big(T/(f_1,\ldots, f_{n-1}); t_1,t_2\Big)
        = \frac{(1-t_1t_2)^{n-1}}{(1-t_1)^n(1-t_2)^n}.$$
    Hence $\cmg{P_n}= (t_1+t_2)^{n-1}$ as desired.
\end{proof}

%%%%%%%%%%%%%%%%%%%%%%%%%%%%%%%%%%%%%%%%%%%%%%%%%%%%%%%%%%%%
\section{Concluding Remarks}\label{sec:conclusion}
%%%%%%%%%%%%%%%%%%%%%%%%%%%%%%%%%%%%%%%%%%%%%%%%%%%%%%%%%%%%
Our studies relied extensively on the computer software \textit{Macaulay2} \cite{M2} because of its efficiency at algebraic geometry computations \footnote{The Macaulay2 code used for these calculations can be found \href{https://github.com/MrBobJrIV/Multidegrees-of-Binomial-Edge-Ideals}{here}: https://github.com/MrBobJrIV/
Multidegrees-of-Binomial-Edge-Ideals.}. Of special importance were the Graphs, EdgeIdeals, and Nauty packages for generating and studying graphs, as well as the Visualize package for convenient viewing of graphs.

One of Macaulay2's most useful contributions was to investigate the proportion of graphs whose multidegrees had certain interesting properties. For example, a surprisingly large portion of graphs have multiplicity-free multidegrees, meaning all coefficients are at most 1. We also examined the leading coefficients of the multidegrees. Almost all graphs we checked have a leading coefficient of 1, which is what motivated us to search for and discover the horned complete graphs in \cref{ex:horned}, proving the existence of a graph whose multidegree has leading coefficient equal to any positive integer $n$. It is still unclear to us what properties of a graph may cause these interesting multidegrees, and we expect to investigate this question in future research. Below is the table of frequencies we collected.

\begin{center}
    \begin{tabular}{cccccccccc}\toprule
        Vertices & 1 & 2 & 3 & 4 & 5 & 6 & 7 & 8 & 9 \\
        \midrule
        Simple connected graphs & 1 & 1 & 2 & 6 & 21 & 112 & 853 & 11,117 & 261,080 \\
        % https://oeis.org/A001349
        Multiplicity-free & 1 & 1 & 1 & 4 & 11 & 60 & 456 & 6,676 & 183,838 \\
        % Missing terms & 0 & 0 & 0 & 0 & 0 & 6 & 46 & 561 & 7,385 \\
        Leading coefficient of 2 & 0 & 0 & 0 & 0 & 0 & 1 & 2 & 24 & 203 \\
        Leading coefficient of 3 & 0 & 0 & 0 & 0 & 0 & 0 & 0 & 0 & 3 \\
        \bottomrule
    \end{tabular}
\end{center}

Notice that the table only goes up to graphs of order 9. This is because the number of calculations needed to study each graph is immense. One improvement we achieved comes from the following proposition, which effectively halves the number of subsets needed to be checked when finding $\minS{G}$.

\begin{proposition}\label{prop:size}
    Let $G$ be a graph. For all $S\in\minS{G}$, $|S| \leq \floor{\frac{n-1}{2}}$.
\end{proposition}
\begin{proof}
    Let $G$ be a graph of order $n$ and let $S\in\minS{G}$. This means for all $T\subseteq V(G)$ we have $\height{P_S(G)} \leq \height {P_T(G)}$. In particular, $\height{P_S(G)} \leq \height{P_{\varnothing}(G)}$. By \cite[Lemma 3.1]{herzog}, $|S|+n-c(\complement{S}) \leq n-1$. Notice that $c(\complement{S})$ is at most $n-|S|$ (the case where all vertices not in $S$ are isolated in $G[\complement{S}]$), so
    \begin{align*}
        |S|+n-(n-|S|) &\leq n-1 \\
        2|S| &\leq n-1 \\
        |S| &\leq \frac{n-1}{2}.
    \end{align*}
    Since $|S|$ must be an integer, we get $|S| \leq \floor{\frac{n-1}{2}}$ as desired.
\end{proof}

Despite this improvement, our machines could not finish the computations for graphs of order 10 after a week of nonstop computing. 

Another area that could be explored further is generalizations of some of the graphs presented in \cref{sec:examples}. For example, it is very natural to consider barbell graphs with more than two `bells', or bells of different sizes. Our methods easily provide the multidegrees of these graphs in terms of products of polynomials, but despite the approachable appearance of these products, we could not find satisfying simplifications or acquire exact expressions for the final coefficients.

%%%%%%%%%%%%%%%%%%%%%%%%%%%%%%%%%%%%%%%%%%%%%%%%%%%%%%%%%%%%
\section*{Acknowledgements}\label{sec:thanks}
%%%%%%%%%%%%%%%%%%%%%%%%%%%%%%%%%%%%%%%%%%%%%%%%%%%%%%%%%%%%
The authors would like to thank their advisors Jonathan Montaño and Zilin Jiang for their impressive knowledge, invaluable mentorship, and inspiring enthusiasm.

\bibliographystyle{plain}
\bibliography{mbei}
\end{document}